\documentclass[12pt,a4paper]{amsart}
\usepackage[colorlinks=true,linkcolor=blue,urlcolor=blue,citecolor=blue]{hyperref}
\usepackage{mlmodern}
\usepackage[dvipsnames]{xcolor} 
\usepackage{amsmath,color}
\usepackage{amsthm, enumerate}
\usepackage{amssymb}
\usepackage[all]{xy}
\usepackage{epsfig, soul}
\usepackage{geometry}
\usepackage{breqn}
\usepackage{amsrefs}

\newtheorem{theorem}{Theorem}[section]
\newtheorem{proposition}[theorem]{Proposition}

\theoremstyle{definition}

\newtheorem{remark}[theorem]{Remark}

\author[K. A. Santos]{Karine de Almeida Santos}

\address{Departamento de Matem\'atica y Ciencias\\
	Universidad de San Andr\'es\\ Vito Dumas, 284\\ Victoria, Provincia de Buenos Aires, Argentina.}
\email{kdealmeidasantos@udesa.edu.ar}

\begin{document}

\title{Dynamics of an isosceles problem generated by a perturbation of Euler's collinear solution}

\begin{abstract}
This paper presents a study of the isosceles problem resulting by a perturbation of Euler's collinear solution under Newtonian gravitational attraction of three bodies in space. After the Hamiltonian was obtained, a circumference of relative equilibria points was found. The original system was subsequently reduced to another system with two degrees of freedom, periodic in the time, where there is now a single point of equilibrium. Linear and parametric stability were discussed in this simplified model of the three-body problem.

\end{abstract}

\keywords{isosceles problem, reduction degree freedom of Hamiltonian system, resonances, parametric stability, boundary curves of stability/instability.
}

\maketitle

\section{Introduction}\label{sec1}
An isosceles solution of the three-body problem is a solution in which the three masses, at any given moment, form the configuration of an isosceles triangle that does not degenerate into an equilateral triangle or a line segment. Since the coordinates of a solution to the three-body problem are analytic functions of time, if a solution is collinear or equilateral for any interval of time, no matter how small, it remains collinear or equilateral, as the case may be, for all time. Thus, such particular configurations only occur for isolated values of time, and consequently, the base of the isosceles triangle is well defined.

We recall that in an isosceles solution of the three-body problem, the masses at the base of the triangular configuration must be equal. Using this information, it is proved that there are exactly three types of isosceles solutions, one planar with zero angular momentum and masses at the base located symmetric with respect to a fixed axis and two spatial, both with nonzero angular momentum. In one of these the masses at the base are located symmetric with respect to the plane through the center of mass and orthogonal to the angular momentum, in the other spatial type the masses at the base are located symmetric with respect to the fixed line through the center of mass parallel to the angular momentum. These facts can be seen in \S\S 344-346 of Wintner's book \cite{wintner} together with Figures 12(i)-(iii) in \S 346. Se also Cabral \cite{hild-isosceles}.

For both the planar case and the spatial case with a symmetry plane, collisions may occur, and when this happens, the motion is defined only during a finite time, unless regularizations are made. We will work with the spatial isosceles solution in which the masses at the base of the triangle move symmetrically to the fixed line defined by the direction of the angular momentum. Along this solution collisions cannot occur and the motion is defined for all time.

\section{Statement of the problem}\label{sec2}

Initially, let us consider the Euler's collinear solution of the three-body problem where mass $m_3$ remains fixed at the origin of the system with null vector position $\overline{r}_3^e$, and the others two masses $m_1$ and $m_2$ with equal value $m$ are moving elliptical and symmetrically on the $xy$-plane with respect to $m_3$, see Figure \ref{fig:Euler's solution}. Owing to symmetry of this problem, the masses $m_1$ and $m_2$ located symmetrically, have their respective vectors position $\overline{r}_1^e$ and $\overline{r}_2^e$, satisfy $\overline{r}_2^e=-\overline{r}_1^e$.

\begin{figure}[!htb]
	\centering
	\includegraphics[scale=.4]{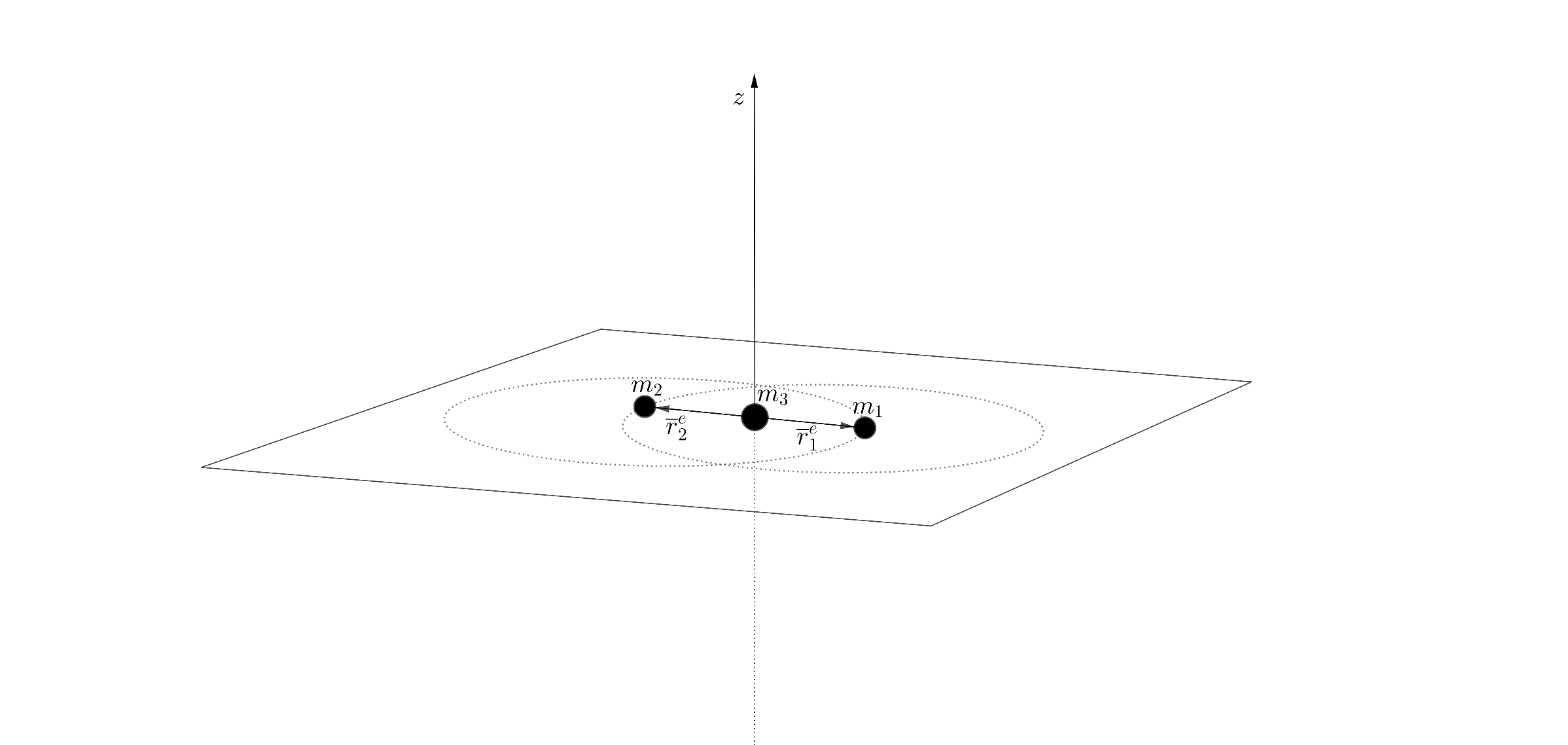}
	\caption{Collinear Euler's solution.}
	\label{fig:Euler's solution}
\end{figure}

The motion of the mass $m_1$ is described by
\begin{equation}\label{kepler}
	\ddot{\overline{r}_1^e}=-\dfrac{G(m+4m_3)}{4\Vert \overline{r}_1^e\Vert^3}\overline{r}_1^e,
\end{equation}
where $G$ is the gravitational constant, $\Vert \cdot \Vert$ represents the usual norm and $\overline{r}_1^e$ it is Kepler's solution on the plane. It is known from his theory that $\overline{r}_1^e=\rho \hat{e}_1$, $\hat{e}_1=(\cos\nu,\sin\nu,0)$, where
\begin{equation}\label{relacao2}
	\rho=\dfrac{p}{1+\epsilon\cos\nu},
	\quad
	\dot{\nu}=\dfrac{c}{\rho^2} 
	\quad \mbox{and} \quad
	\ddot{\rho}=\dfrac{c^2}{\rho^3}\dfrac{\epsilon\cos\nu}{(1+\epsilon\cos\nu)}.
\end{equation}
Here, $\rho=\Vert \overline{r}_1^e \Vert$ represents the distance of the attractor focus, $\nu$ is the true anomaly, $\epsilon$ is the eccentricity, $c$ is area constant and $p=a(1-\epsilon^2)$ is the orbit parameter, where $a$ is the major semiaxis of the ellipse.

Next we will build the spatial isosceles solution from the perturbation of the vectors $\overline{r}_1^e$ and $\overline{r}_2^e$ of the Euler's collinear solution. For this, consider $\overline{v}=(v_1,v_2,0)$ and $\overline{w}=(0,0,w_3)$ two vectors that perform the perturbation in the Euler's collinear solution as follows
\begin{eqnarray}\label{r1r2r3}
	r_1=(\overline{r}_1^e+\overline{v})+\overline{w}
	\qquad  \mbox{and} \qquad
	r_2=-(\overline{r}_1^e+\overline{v})+\overline{w},
\end{eqnarray}
and illustrated by Figure \ref{fig:perturbation}.
Consequently, to ensure that the center of mass remains at the origin, the position vector of $m_3$ is given by:
\begin{equation}\label{r3}
	r_3=-\frac{2m}{m_3}\overline{w}.
\end{equation}

\begin{figure}[!htb]
	\centering
	\includegraphics[scale=.4]{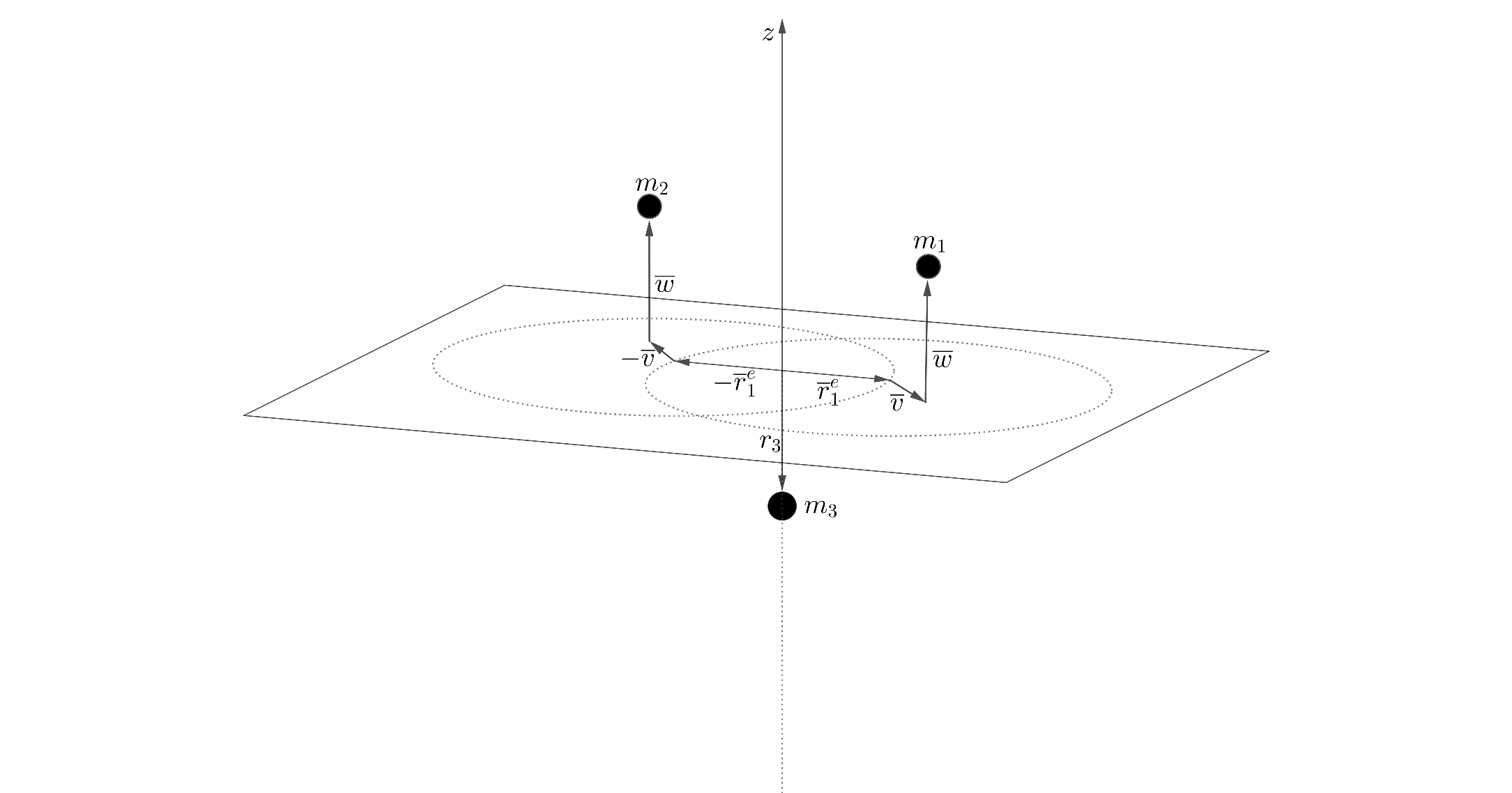}
	\caption{Perturbation of Euler's collinear solution.}
	\label{fig:perturbation}
\end{figure}

\begin{proposition}
	There are vector functions $\overline{v}=\overline{v}(t)$ and $\overline{w}=\overline{w}(t)$ such that the vectors $r_1,\;r_2$ and $r_3$ at \eqref{r1r2r3} and \eqref{r3} are three-body problem solutions.
\end{proposition}

\begin{proof} 
	Like $r_1$, $r_2$ and $r_3$ are three-body problem solutions, that is, they satisfy the equations
	\begin{eqnarray}\label{3corpos}
		\ddot{r_1}&=&\frac{Gm}{d_{12}^3}(r_2-r_1)+\frac{Gm_3}{d_{13}^3}(r_3-r_1)\nonumber\\
		\ddot{r_2}&=& \frac{Gm}{d_{12}^3}(r_1-r_2)+\frac{Gm_3}{d_{23}^3}(r_3-r_2)\\
		\ddot{r_3}&=&\frac{Gm}{d_{13}^3}(r_1-r_3)+\frac{Gm}{d_{23}^3}(r_2-r_3),\nonumber
	\end{eqnarray}
	where	
	$$d_{12}=\Vert r_1-r_2\Vert,
	\quad
	d_{13}=\Vert r_1-r_3\Vert
	\quad \mbox{and} \quad
	d_{23}=\Vert r_2-r_3\Vert
	.$$
	Using \eqref{r1r2r3} and \eqref{r3}, we can write $r_1-r_2=2(\overline{r}_1^e+\overline{v})$ and
	$r_1+r_2-2r_3=2\dfrac{M}{m_3}\overline{w}$, where $M=2m+m_3$ it is the total mass. Furthermore, as $m_1$ and $m_2$ remain equidistant from $m_3$, we have $d_{13}=d_{23}$. Therefore \eqref{3corpos} can be rewritten as follows
	
	\begin{equation}\label{sisperturbkepler}
		\begin{array}{rl}
			\ddot{\overline{r}_1^e}+\ddot{\overline{v}}&=-G\left( \dfrac{2m}{\overline{d}_{12}^3}+\dfrac{m_3}{\overline{d}_{13}^3}
			\right)(\overline{r}_1^e+\overline{v})\\
			\ddot{\overline{w}}&=-\dfrac{GM}{\overline{d}_{13}^3}\overline{w},
		\end{array}
	\end{equation}
	where
	$$
	\overline{d}_{12}=\Vert 2(\overline{r}_1^e+\overline{v}) \Vert
	\quad \mbox{and} \quad
	\overline{d}_{13}=\Big\Vert (\overline{r}_1^e+\overline{v})+\dfrac{M}{m_3}\overline{w}\Big\Vert.
	$$
	The vector function $\overline{r}_1^e(t)$ is a solution to Kepler's problem \eqref{kepler} and describes the motion of mass $m_1$ in the considered Euler's collinear solution. Note that $\overline{v}=0$, $\overline{w}=0$ is an equilibrium solution of this system, because as mentioned, $\overline{r}_1^e$ is a solution to Kepler's problem. Thus, \eqref{sisperturbkepler} is a second-order and non-autonomous system in $\overline{v}$, $\overline{w}$ with analytical data. Therefore, by the existence and uniqueness theorem in ordinary differential equations, for each initial condition $\overline{v}_0$, $\overline{w}_0$, $\dot{\overline{v}}_0$, $\dot{\overline{w}}_0$, there is only one solution to this system defined on an open time interval, whose coordinates are analytical functions on time. As it determines the vectors in \eqref{r1r2r3} that form an isosceles configuration in this interval, by analyticity, the solution \eqref{r1r2r3} coincides with the solution isosceles at all times it is defined, and because it is of the third type, and therefore it exists for all time. 
\end{proof}

The equations \eqref{sisperturbkepler} define a time periodic Hamiltonian system with three degrees of freedom. As will be seen in section \ref{sec6} it will be reduced to a time periodic Hamiltonian system with two degrees of freedom which describes a simplified model of the three-body problem and have one relative equilibrium. The reduced system is still periodic and is time dependent, so we can not use Arnold's theorem to study the stability of this equilibrium. Therefore in section \ref{sec7} we  will study the conditional stability, that is we will study the parametric stability of this equilibrium point.

\section{Hamiltonian in rotating-pulsating coordinates}\label{sec3}

In this section we consider a rotating coordinates system with angular speed given by the true anomaly of the elliptic motion. Together with the rotating coordinates we introduce pulsating coordinates via the radius vector of the elliptic motion. 
After these coordinates change, we take the true anomaly $\nu$ as time variable. Proceeding with the computations we arrive at an expression of the Hamiltonian in these coordinates. 

First, to obtain a simpler equation for the problem, let us regularize the measurement units. Applying a new scale in a vertical vector, $\overline{w}$, with $w_3=\sqrt{\dfrac{m_3}{M}}v_3$, and denoting $v=(v_1,v_2,v_3)$,  \eqref{sisperturbkepler} becomes :

\begin{equation}\label{kepler perturb}
	\ddot{v}=\nabla_v U(v,t,m,m_3)-\ddot{\overline{r}_1^e}, %\qquad\mbox{where}\qquad r_1^e=\rho \hat{e}_1, \quad %\hat{e}_1=(\cos\nu,\sen\nu,0)
\end{equation}
where the potential function is $U(v,t,m,m_3)=\dfrac{Gm}{4d_{12}}+\dfrac{Gm_3}{d_{13}},$ with $d_{12}=\Vert \overline{r}_1^e+\overline{v}\Vert$ and $d_{13}=\Big\Vert (\overline{r}_1^e+\overline{v})+\sqrt{\dfrac{M}{m_3}}v_3e_3\Big\Vert$.
Then, in \eqref{kepler perturb}, denote $\mu=\dfrac{m}{m_3}$, consider $l$ as a length unit such that $\dfrac{Gm_3}{l^3}=1$ and put
$$
l\overline{r}_1^e=r_1^e
\qquad \mbox{and}\qquad 
lv=z,\qquad \mbox{where} \qquad
z=(z_1,z_2,z_3).
$$
Finally, due to the homogeneity of the potential function and because $\overline{r}_1^e=\rho(\cos\nu,\sin\nu,0)$, we obtain 
\begin{equation}\label{ode 3bp}
	\ddot{z}=\nabla U(z,t,\mu)-\ddot{r_1^e},
\end{equation}
where the potential function is now
$$U(z,t,\mu)=\dfrac{\mu}{4d_{12}}+\dfrac{1}{d_{13}},$$
with 
\begin{eqnarray*}
	d_{12}&=&\sqrt{(\rho\cos\nu+v_1)^2+(\rho\sin\nu+v_2)^2},\\
	d_{13}&=&\sqrt{(\rho\cos\nu+v_1)^2+(\rho\sin\nu+v_2)^2+(2\mu+1)v_3^2}.
\end{eqnarray*}

Remember that $\hat{e}_1=(\cos\nu,\sin\nu,0)$, then $r_1^e=\rho\hat{e}_1$ and
$
\ddot{r}_1^e=(\ddot{\rho}-\rho\dot{\nu}^2)\hat{e}_1=(\ddot{\rho}-\rho\dot{\nu}^2)\Omega e_1$, where $e_1$ it is a canonical vector of $\mathbb{R}^3$, $\Omega(\nu)=\Omega$ is the rotation matrix given by $\Omega=
\begin{bmatrix}
	\cos\nu  &- \sin\nu  & 0  \\
	\sin\nu  & \cos\nu  & 0\\
	0        &    0     & 1
\end{bmatrix}.$
Furthermore, using  second and third Kepler's relation in \eqref{relacao2}, we have
\begin{equation*}
	\ddot{r}_1^e=-\dfrac{c^2}{\rho^3}\dfrac{1}{1+\epsilon\cos\nu}\Omega e_1.
\end{equation*}

Let us pass the vector $z$ to the rotating-pulsating coordinates system, as commented at the beginning of this section. Applying the first change of coordinate given by rotation, $z=\Omega \zeta$, we compute:
\begin{eqnarray}\label{newexpression}
	\ddot{z}&=&\Omega[
	\ddot{\zeta}+2\dot{\nu}\Sigma\dot{\zeta}+\ddot{\nu}\Sigma 
	\zeta+\dot{\nu}^2\Sigma^2\zeta],
\end{eqnarray}
where the prime represents the derivative with respect to true anomaly and $\Omega'=\Omega\Sigma$ and $\Omega''=\Omega\Sigma^2$, with $\Sigma=\begin{bmatrix}
	0 & -1  & 0  \\
	1 & 0  & 0\\
	0 & 0  & 0
\end{bmatrix}
$.

The second change of coordinates defined in terms of the radius vector $\rho$ at \eqref{relacao2} is given by $\zeta=\rho x$. Using the chain rule we obtain
\begin{eqnarray*}
	\dot{\zeta}= \dot{\rho}x+\rho\dot{\nu}x'\label{d1x}
	\qquad \mbox{and} \qquad
	\ddot{\zeta}=\ddot{\rho}x+\left(2\dot{\rho}\dot{\nu}+\rho\ddot{\nu}\right)x'+\rho\dot{\nu}^2x''.
\end{eqnarray*}
Because $\dot{\nu}\rho^2=c$ follows that
$
2\dot{\rho}\dot{\nu}+\rho\ddot{\nu}=0, 
$
so the above equation is now $\ddot{\zeta}=\ddot{\rho}x+\rho\dot{\nu}^2x''$, and \eqref{newexpression} is given by
\begin{equation*}
	\ddot{z}=\Omega\left[\varrho\dot{\nu}^2(x''+2\Sigma x')+(\ddot{\rho}I+\dot{\nu}^2\rho\Sigma^2)x)\right],
\end{equation*}
where $I$ is identity matrix of order three.

Thus, using the second and third relation in \eqref{kepler}, the above equation is 
\begin{equation*}
	\ddot{z}=\dfrac{c^2}{\rho^3}\Omega\left[
	x''+2\Sigma x'+\Sigma^2x+\dfrac{\epsilon\cos\nu}{1+\epsilon\cos\nu}x
	\right].
\end{equation*}
Finally, using the rotation invariance and the homogeneity of the potential function we have $\nabla_zU(z,t,\mu)=\dfrac{1}{\rho^2}\Omega \nabla_x W(x,\mu)$, where $W(x,\mu)=\dfrac{\mu}{4d_1}+\dfrac{1}{d_2},$
with
$
d_{1}=\sqrt{(1+x_1)^2+x_2^2} \quad\mbox{and}\quad
d_{2}=\sqrt{(1+x_1)^2+x_2^2+(2\mu+1)x_3^2}.
$
Then, \eqref{ode 3bp} can be rewritten as
\begin{equation*}
	x''+2\Sigma x'+\Sigma^2x=\dfrac{1}{1+\epsilon\cos\nu}\left(-\epsilon\cos\nu\, x+\dfrac{p}{c^2}\nabla_x W(x,\mu)+e_1\right).
\end{equation*}
From the first and third equations in \eqref{relacao2} we have $\displaystyle{\dfrac{\rho}{c^2}= \frac{1}{\kappa(1+\epsilon \cos \nu)}}$,
so by defining $y=x'+\Sigma x$, we obtain the following first order system
\begin{equation}\label{pulsante}
	x'=y-\Sigma x \qquad
	y'=-\Sigma y+\nabla_x V(x,\nu,\mu,\epsilon),
\end{equation}
whose Hamiltonian $H: \Gamma\times\mathbb{R}^3\times \mathbb{R}\times I\times [0,1)\longrightarrow \mathbb{R}$ is given by
\begin{equation}\label{hinicial}
	H(x,y,\nu,\mu,\epsilon)=\dfrac{1}{2}\Vert y \Vert^2-V(x,\nu,\mu,\epsilon)-\langle \Sigma x,y \rangle,
\end{equation}
where $\Gamma=\{(x_1,x_2,x_3)\in\mathbb{R}^3; x_1\neq -1 \mbox{ or }x_2\neq0\}$, $I$ is a subset of positive real numbers and
\begin{equation}\label{potencialgeral}
	V(x,\nu,\mu,\epsilon)
	=\dfrac{1}{1+\epsilon\cos\nu}\left(-\dfrac{1}{2}\epsilon\cos\nu\Vert x\Vert^2+\dfrac{1}{\kappa}W(x,\mu)+x_1\right),
	\quad \kappa=\frac{\mu+4}{4}.
\end{equation}

\section{Relative Equilibria}\label{sec4}

\begin{proposition}
	The system \eqref{pulsante} admits a family of equilibria points in the form $z^*=(x^*,y^*)=((1+x_1^*),x_2^*,0,-x_2^*,(1+x_1^*),0)$, where $((1+x_1)^*,x_2^*)$ belongs to the unitary radius circumference centered at $(-1,0)$ in the plane $x_1x_2$.
\end{proposition}
\begin{proof}
	The equilibria points are obtained by solving $(x',y')=(0,0)$ for all $\nu\in\mathbb{R}$, $\mu\in I$, and $\epsilon\in[0,1)$. That is, by \eqref{pulsante}, we have
	\begin{eqnarray}
		y&=&\Sigma x \label{eq1perturb}\\
		\nabla_{x} V(x,\nu,\mu,\epsilon)&=&\Sigma^2x\label{eq2perturb},
	\end{eqnarray}
	for all $\nu\in\mathbb{R}$, $\mu\in I$, and $\epsilon\in[0,1)$.
	Additionally, 
	$$
	\nabla_{x}V(x,\nu,\mu,\epsilon)=\dfrac{1}{1+\epsilon\cos\nu}\left(-\epsilon\cos\nu x+\dfrac{1}{\kappa}\nabla_{x}W(x,\mu)+e_1\right),
	$$
	where
	\begin{equation}\label{parciasW}
		\nabla_{x}W(x,\mu)=-\left(B(x,\mu)(1+x_1),B(x,\mu) x_2,\dfrac{(2\mu+1)}{d_{2}^3}x_3\right),
	\end{equation}
	with
	$$B(x,\mu)=\dfrac{\mu}{4d_{1}^3}+\dfrac{1}{d_{2}^3}.$$
	
	Since the third entry of $\Sigma^2 x$ is zero, the third equation of \eqref{eq2perturb} is satisfied for any $\nu\in\mathbb{R}$, $\mu\in I$, and $\epsilon\in[0,1)$ if, and only if, $x_3=0$.
	When this occurs, we have $d_{1}=d=d_{2}=\sqrt{(1+x_1)^2+x_2^2}$. Thus, the first two equations of the system \eqref{eq2perturb} can be rewritten as follows:
	\begin{eqnarray*}
		\dfrac{1}{1+\epsilon\cos\nu}
		\left(
		1-\dfrac{1}{\kappa}\dfrac{\kappa}{d^3}
		\right)(1+x_1)=0,\qquad
		\dfrac{1}{1+\epsilon\cos\nu}
		\left(
		1-\dfrac{1}{\kappa}\dfrac{\kappa}{d^3}
		\right)x_2=0.
	\end{eqnarray*}
	
	Note that $1+x_1$ and $x_2$ cannot be simultaneously zero, as it would lead to a singularity for $W$. Therefore, by both of the above equations, we have $d^3=1$, i.e.,
	\begin{equation}\label{cir-eq}
		(1+x_1)^2+x_2^2=1, \;\text{for all} \; \nu\in\mathbb{R},\; \mu\in I \text{ and } \epsilon\in[0,1).
	\end{equation}
	Now, because \eqref{eq1perturb}, we have $y_1=-x_2$, $y_2=1+x_1$. Finally, we denote these equilibria as in the statement. 
\end{proof}
\begin{remark}
	Previously, we observed that the equilibrium $\overline{v}=0$, $\overline{w}=0$ of the system \eqref{sisperturbkepler} gives rise to Euler's collinear solution we generate the isosceles solutions. Geometrically, because of rotation and homothety to each point taken on the equilibria circumference \eqref{cir-eq}, we generate an ellipse whose focus is common to the center $(-1,0)$ of this circumference, and whose major axis is defined as $\nu$ varies. See this interpretation in Figure \ref{fig:circeq} below. Thus, each equilibria of Hamiltonian system \eqref{hinicial} situated on circumference \eqref{cir-eq} corresponds to an elliptical collinear solution. Note that the equilibria on dynamical of isosceles movements are not isosceles solutions.
\end{remark}

\begin{figure}[!htb]
	\centering
	\includegraphics[scale=.4]{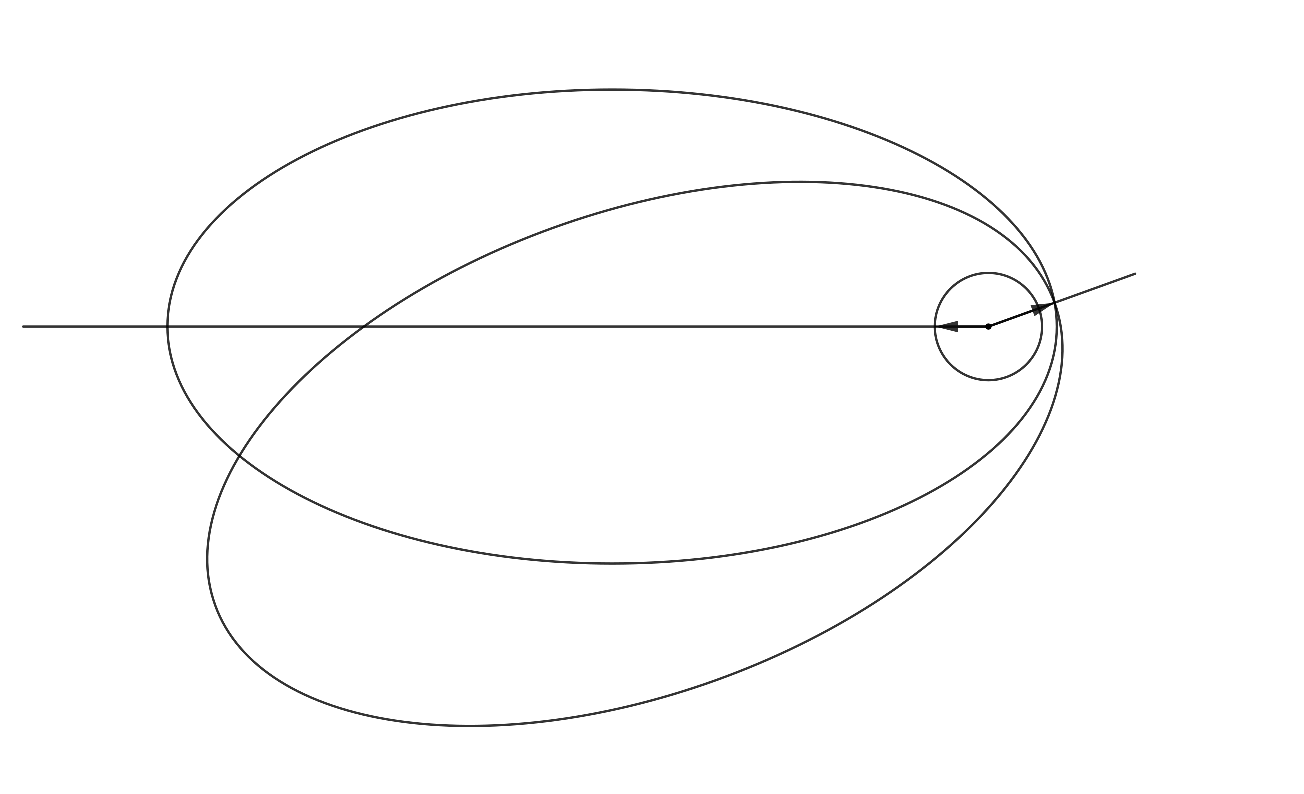}
	\caption{\textmd{Geometric interpretation of equilibria circumference.}}
	\label{fig:circeq}
\end{figure}

\section{Characteristic equation of Hamiltonian system}

Setting $z=(x,y)$, the linearized system at an equilibrium point can be write in the form $z' = A(\nu, \mu, \epsilon)z$, where the Hamiltonian matrix is given by
$$
A(\nu, \mu, \epsilon) = JH_{zz}(z^*, \nu, \mu, \epsilon) =
\begin{bmatrix}
	-\Sigma        & I \\
	V_{xx}^* & -\Sigma
\end{bmatrix},
$$
where $J=\begin{bmatrix}
	O & I\\
	-I & O
\end{bmatrix}$,
$I$ is the identity matrix of order three, $V_{xx}^* = V_{ xx}(x^*, \nu, \mu, \epsilon) = \dfrac{1}{1+\epsilon\cos\nu}\left(-\epsilon\cos\nu I+\dfrac{1}{\kappa}W_{ x x}^*(x,\mu)\right)$, and $W_{xx}^* = W_{xx}(x^*, \mu)$ denotes the Hessian matrix of $W$ at equilibrium. When $\epsilon=0$, we have an autonomous system represented by the Hamiltonian matrix $A=A(\mu)$, where $V_{ xx}^*=\dfrac{1}{\kappa}W_{ x x}^*$. The entries of this matrix, can be obtained using \eqref{parciasW}, then
\begin{eqnarray*}
	W_{x_1x_1}&=&3B_{1}(x,\mu)(1+x_1)^2-B(x,\mu)
	\qquad
	W_{x_1x_2}=3B_{1}(x,\mu)(1+x_1)x_2
	\\
	W_{x_2x_2}&=&3B_{1}(x,\mu)x_2^2-B(x,\mu)
	\qquad \qquad \quad
	W_{x_3x_3}=\dfrac{(2\mu+1)}{d_{2}^3}\left(\dfrac{3(2\mu+1)}{d_{2}^2}x_3^2-1\right),
\end{eqnarray*}
where $B_{1}(x,\mu)=\left(\dfrac{\mu}{4d_{1}^5}+\dfrac{1}{d_{2}^5}\right)$.
When $x_3^*=0$, and \eqref{cir-eq} is satisfied, the entries of the matrix $W_{xx}^*$ are as follows:
\begin{eqnarray*}
	W_{x_1x_1}^*=a&=&\kappa\left[3(1+{x_1^*})^2-1\right],\qquad
	W_{x_1x_2}^*=b=\kappa\left[3(1+x_1^*)x_2^*\right],\\
	W_{x_2x_2}^*=c&=&\kappa\left[3{x_2^*}^2-1\right],\hspace{1.8cm}
	W_{x_3x_3}^*=d=-(2\mu+1),
\end{eqnarray*}
with the remaining entries being zero.
The characteristic polynomial of $A$ is given by
\begin{eqnarray}
	p(\lambda)&=&
	\left(\frac{d}{\kappa}-\lambda^2\right)\left[\lambda^4+\left(2-\frac{a+c}{\kappa}\right)\lambda^2+\left(\frac{a}{\kappa}+1\right)\left(\frac{c}{\kappa}+1\right)-\frac{b^2}{\kappa^2}\right].\nonumber
\end{eqnarray}
Now, as $(x_1^*,x_2^*)$ belongs to the equilibria circumference, it follows that
\begin{eqnarray*}
	2-\frac{a+c}{\kappa}=1\qquad
	\left(\frac{a}{\kappa}+1\right)\left(\frac{c}{\kappa}+1\right)-\frac{b^2}{\kappa^2}=0.
\end{eqnarray*}
Therefore, replacing the value of $d$, the characteristic polynomial is given by:
\begin{equation*}
	p(\lambda) =-\lambda^2\left(\dfrac{2\mu+1}{\kappa}+\lambda^2\right)(\lambda^2+1).
\end{equation*}

Thus, we have a double-zero eigenvalue and two distinct purely imaginary eigenvalues when $\mu \neq 0$.

Note that the linearized Hamiltonian at the equilibrium solution when $\epsilon=0$ is given by
\begin{eqnarray}\label{Hleq}
	H_0(x,y,\mu)&=&\dfrac{1}{2}\left[
	\Vert y \Vert^2-\dfrac{1}{\kappa}\langle x,W_{xx}^*x\rangle-2\langle \Sigma x, y\rangle \right],
	\quad \kappa=\dfrac{\mu+4}{4}.
\end{eqnarray}

\section{Null eigenvalue elimination and normal form of the reduced Hamiltonian}\label{sec6}
As we have seen, the Hamiltonian matrix $A$ has a zero eigenvalue with a multiplicity equal to two. This occurs because the function
$
\mathcal{Q}(x,y)=-\langle \Sigma x,y\rangle=x_2y_1-x_1y_2,
$
is a first integral of the system \eqref{hinicial} as is proven in the following result.
\vspace{.3cm}

\begin{proposition}\label{prop-reduction}
	Let $x_0$ be an equilibrium of $\dot{x}=f(x)$, $x\in\mathbb{R}^n$, and $\Psi$ be a first integral of this system defined in a neighborhood of $x_0$ such that $\nabla \Psi(x_0)\neq 0$. Then $A=Df(x_0)$ has a zero eigenvalue.
\end{proposition}
\begin{proof}
	By contradiction.
\end{proof}
The following result can be seen in \cite[Section~18]{siegel}.

\vspace{.3cm}
\begin{proposition}\label{prop-reduce-dof}
	Consider an autonomous Hamiltonian system $\dot{z}=J\nabla H(z)$, where $\Psi$ is a time-independent first integral. Then, in an open set where $\nabla\Psi(z)\neq0$, it is possible to decrease the number of degrees of freedom of the system by one unit.
\end{proposition}
\vspace{.3cm}

\begin{proposition}\label{dfr}
	There exists a symplectic transformation such that the Hamiltonian in \eqref{Hleq}, with three degrees of freedom, is taken to the following non-autonomous Hamiltonian with two degrees of freedom
	\begin{equation}\label{hamil-reduzido}
		\mathcal{H}(u_1,u_3,v_1,v_3,\nu,\mu,\gamma,\epsilon)
		=
		\frac{1}{2}\left(v_1^2+v_3^2
		\right)
		+\mathcal{V}(u_1,u_3,v_1,v_3,\nu,\mu,\gamma,\epsilon),
	\end{equation}
	with $u_1\neq0$ and the potential \eqref{potencialgeral} now given by
	$$
	\mathcal{V}(u_1,u_3,\nu,\mu,\gamma,\epsilon)=
	\frac{\gamma^2}{2u_1^2}
	+\frac{1}{1+\epsilon\cos\nu}\left(
	\frac{1}{2} \epsilon\cos\nu(u_1^2+u_3^2)-\frac{1}{\kappa}\mathcal{W}(u_1,u_3,\mu)
	\right),
	$$
	where 
	$$
	\mathcal{W}(u_1,u_3,\mu)=\left(
	\frac{\mu}{4|u_1|}+\frac{1}{\sqrt{u_1^2+(2\mu+1)u_3^2}}
	\right).
	$$
	
\end{proposition}
\begin{proof}
	Because of the Proposition $\ref{prop-reduction}$ and the process described in proof of Proposition $\ref{prop-reduce-dof}$ we can get the  following symplectic transformation given by
	\begin{equation*}
		\begin{array}{lcr}
			x_1=u_1\cos u_2 \quad & \quad x_2=-u_1\sin u_2 & \quad x_3=u_3,\\
			y_1=v_1\cos u_2-\dfrac{v_2}{u_1}\sin u_2 \quad & \quad 	y_2=-v_1\sin u_2-\dfrac{v_2}{u_1}\cos u_2 & \quad	y_3=v_3,
		\end{array}
	\end{equation*}
	generated by 
	$
	\mathcal{F}(u,y)=\left(u_1\cos u_2\right)y_1+\left(-u_1\sin u_2\right)y_2+u_3y_3,
	$ since $u_1\neq0$.
	
	\noindent Due to above mentioned transformation, now $\mathcal{Q}=v_2$ is a constant of motion, say $v_2=\gamma$. That is, the term $\mathcal{Q}$ in \eqref{Hleq} can be ignored and so we obtain the reduced Hamiltonian as \eqref{hamil-reduzido}.
	
\end{proof}

If $\gamma\neq1$, then \eqref{hamil-reduzido} there is no equilibria. If $\gamma=1$, its only equilibrium is the point
$$
P^{*}: u_1=1, \quad u_3=0, \quad v_1=0, \quad v_3=0.
$$

In order to study the stability of the equilibrium point $P^*$ we change to new coordinates $(\xi_1,\xi_2,\eta_1,\eta_2)$ with its origin at the equilibrium point via $$u_1=\xi_1+1,\qquad 
u_3=\xi_2, \qquad 
v_1=\eta_1, \qquad 
v_3=\eta_2.$$
Then, the Hamiltonian \eqref{hamil-reduzido} becomes
\begin{equation*}
	H(\xi_1,\xi_2,\eta_1,\eta_2,\nu,\mu,\gamma,\epsilon)=\dfrac{1}{2}(\eta_1^2+\eta_2^2)+S(\xi_1,\xi_2,\eta_1,\eta_2,\mu,\gamma,\epsilon),
\end{equation*}
where $S(\xi_1,\xi_2,\eta_1,\eta_2,\nu,\mu,\gamma,\epsilon)=\mathcal{V}(\xi_1+1,\xi_2,\eta_1,\eta_2,\mu,\gamma,\epsilon)$.
Thus, the expansion of \eqref{h-red-linearizado} at the origin is given by
\begin{equation}\label{H eq exp}
	H(\zeta,\nu,\mu,\epsilon)=H_0(\zeta,\nu,\mu,\epsilon)+H_1(\zeta,\nu,\mu,\epsilon)+H_2(\zeta,\nu,\mu,\epsilon)+\ldots,
\end{equation}
where each $H_k$, $k=0,1,2,\ldots$ it is a homogeneous polynomial of degrees $k+2$ in the coordinates of $\zeta=(\xi_1,\xi_2,\eta_1,\eta_2)$ and is $2\pi$ periodic in $\nu$ besides this terms containing the parameters $\mu$ and $\epsilon$ and the quadratic term is given by
\begin{eqnarray}\label{h-red-linearizado}
	H_0(\xi_1,\xi_2,\eta_1,\eta_2,\mu,\epsilon)
	&=&\frac{1}{2}\left(\eta_1^2+\eta_2^2\right)+\frac{3}{2}\xi_1^2\nonumber\\
	&-&\frac{1}{1+\epsilon\cos\nu}\left[
	-\frac{\epsilon\cos\nu}{2}(\xi_1^2+\xi_2^2)+\xi_1^2-\frac{(2\mu+1)}{2\kappa}\xi_2^2
	\right].
\end{eqnarray}

\begin{remark}
	By reducing the dynamics to the space obtained when the value of the first integral is fixed, which gave rise to the null eigenvalue, we obtain a simplified system from which the null eigenvalue has been eliminated.
	
\end{remark}

\section{Parametric stability}

In the previous section, although \eqref{hamil-reduzido} has been reduced to another periodic Hamiltonian with two degrees of freedom, it is still time dependent and, as was commented previously, the study of equilibrium stability is not trivial. However, there is possible to study the stability of the equilibrium for the linearized Hamiltonian system in a neighborhood of it. The parametric stability study of the linearized Hamiltonian system can be facilitated by the normal form of the unperturbed linearized Hamiltonian system \eqref{h-red-linearizado}, when $\epsilon=0$, as follows.

Proceeding as \cite[Section 1.2]{markeev} for $\epsilon=0$, the following symplectic linear transformation 
\begin{equation*}
	\xi_1=\omega_1^{-\frac{1}{2}}y_1,\qquad
	\xi_2=\omega_2^{-\frac{1}{2}}y_2,\qquad
	\eta_1=-{\omega_1}^{\frac{1}{2}}x_1,\qquad
	\eta_2=-{\omega_2}^{\frac{1}{2}}x_2,
\end{equation*}
where $\omega_1=1$, $\omega_2=\sqrt{\dfrac{2\mu+1}{\kappa}}$ and $\kappa=\dfrac{\mu+4}{4}$, brings \eqref{H eq exp} to following normalized Hamiltonian
\begin{equation}\label{hexpandidoeps}
	H(x,y,\nu,\mu,\epsilon)=H_0+\epsilon H_1+\dfrac{\epsilon^2}{2!}H_2+\ldots+\dfrac{\epsilon^k}{k!}H_k+\mathcal{O}(\epsilon^{k+1}),
\end{equation}
where
\begin{eqnarray*}
	H_0&=&\dfrac{\omega_1}{2}(x_1^2+y_1^2)+\dfrac{\omega_2}{2}(x_2^2+y_2^2) \\
	H_k&=&\dfrac{(-1)^{k+1}}{2}\left[3 y_1^2- \dfrac{7}{2}\dfrac{\mu}{\mu+4}\dfrac{1}{\omega_2}  y_2^2\right] \cos^k (\nu ),\quad k=1,2,\cdots .
\end{eqnarray*}
For the unperturbed system, which in this case is also autonomous, because the quadratic form $H_0$ is positive definite, the equilibrium $(0,0)$ is stable by the Dirichlet theorem.
For $\epsilon\neq 0$ the system it is time dependent and the question of stability of an equilibrium is not trivial. When the linear system $z' = A(\nu, \mu^*, 0)z$ has a multiple multiplier for some value $\mu^* \in I$, we say that this is the value of \textit{parametric resonance}. According to the Krein-Gelfand-Lidskii (KGL) theorem enunciated below, the unperturbed system for this resonance value can be strongly stable for small values of $\epsilon$, and consequently parametically stable. In this last case, we have that for $(\mu, \epsilon)$ in the neighborhood of $(\mu^*, 0)$, the perturbed system $z' = A(\nu, \mu, \epsilon)z$ is stable for small $\epsilon$. Now, let us assume that the unperturbed system is not strongly stable. That is, in the neighborhood of this system, there are both stable and unstable systems. Furthermore, in the family of parametric systems $z' = A(\nu, \mu, \epsilon)z$, we can have parameters that provide stable systems and parameters that provide unstable systems. Thus, perhaps the parameters can be separated by continuous curves that limit, in the parameter plane $(\mu,\epsilon)$, the regions of the family of systems that are stable from those that are unstable. We then look for the resonances of this system according to the KGL theorem enunciated below that can be consulted in \cite[Section 7.7]{hl}, \cite[Section~3.1]{markeev} and \cite{yaku}, for this specific case of resonance see \cite{mark3}.

Consider a Hamiltonian function of the form $H=H_0+\epsilon H_1+\epsilon^2 H_2+\mathcal{O}(\epsilon^3)$, where $H_1$, $H_2,\ldots$ are quadratic forms in the variables $x_1,x_2,\ldots,x_n$, $y_1,y_2,\ldots,y_n$ with continuous and $2\pi$-periodic coefficients in $t$ and $H_0=\dfrac{1}{2}\displaystyle\sum_{k=1}^{n}\sigma_k(x_k^2+y_k^2)$, where $\sigma_k=\delta_k\omega_k$, with $\delta_k=-1$ or  $\delta_k=1$.
\begin{theorem}
	(Krein-Gelfand-Lidskii: KGL) For sufficiently small $\epsilon$, the linear system with the above Hamiltonian is strongly stable if and only if the values of $\sigma_j$ are not related by equalities of the form $\sigma_k+\sigma_l=N$, where $k,\;l$ are non-negative integers and $N\in\mathbb{Z}$.
\end{theorem}

Note that $2\omega_1(\mu)=N_1=2$ is an integer for all $\mu>0$ and 
\begin{equation}\label{desig ressonancia}
	2<2\omega_2(\mu)<4\sqrt{2}.
\end{equation}
Then, a basic resonance is obtained since $2\omega_2$ is not an integer. On the other hand, double resonances occur when $2\omega_2=N_2$ is integer in interval defined by \eqref{desig ressonancia}, that is $N_2=3,4,5$, and 
the following parametric resonance values are
\begin{equation}\label{parametric values}
	\mu_3^*=\dfrac{20}{23},
	\qquad
	\mu_4^*=3
	\qquad 
	\mbox{and}
	\qquad
	\mu_5^*=12.
\end{equation}

\subsection{Boundary curves of the stability/instability regions}\label{subsec6.3}

Now, let us construct the curves that delimit the stability and instability regions for the case of double resonance. Such curves are built in the parameter plane $(\epsilon,\mu)$ and can be expressed through the following expansion in $\epsilon$:
\begin{equation}\label{curvamugeral}
	\mu=\mu_0+\mu_1\epsilon+\mu_2\epsilon^2+ \mathcal{O}(\epsilon^3),
\end{equation}
where $\mu_0=\mu_{N_2}^*$ it is parametric value that gives rise to cases of double resonance, with $N_2=3,4,5$.

By inserting \eqref{curvamugeral} into \eqref{hexpandidoeps}, and rearranging in powers of $\epsilon$, we get the expansion
\begin{equation}\label{hmuexp}
	H(\xi,\eta,\nu,\mu,\epsilon)=H_0(\xi,\eta,\nu,\mu_0)+\epsilon H_1(\xi,\eta,\nu,\mu_0,\mu_1)+
	\epsilon^2 H_2(\xi,\eta,\nu,\mu_0,\mu_1,\mu_2)+\cdots.
\end{equation}

In order to apply Deprit–Hori method in Kamel’s formulation, see \cite{gerson, lucia, laud, marke2, markeev, mark3}, it is convenient to eliminate the resonant harmonic oscillator making the double rotation
\begin{equation}\label{rotation}
	\begin{array}{ccl}
		x_1&=&\cos\left(\dfrac{N_1\nu}{2}\right)X_1+\sin\left(\dfrac{N_1\nu}{2}\right)Y_1\\
		
		y_1&=&-        
		\sin\left(\dfrac{N_1\nu}{2}\right)X_1+\cos\left(\dfrac{N_1\nu}{2}\right)Y_1\\
		x_2&=&\cos\left(\dfrac{N_2\nu}{2}\right)X_2+\sin\left(\dfrac{N_2\nu}{2}\right)Y_2\\
		
		y_2&=&-    
		\sin\left(\dfrac{N_2\nu}{2}\right)X_2+\cos\left(\dfrac{N_2\nu}{2}\right)Y_2,
	\end{array}
\end{equation}
where $ N_1=2$ and $N_2=3,4,5$. Then, after applying the rotation given by \eqref{rotation} into \eqref{hmuexp}, we obtain a new Hamiltonian $\mathbf{H}(X,Y)=H(X,Y)-W_\nu$, where due to $W_\nu=-H_0$, the term $\mathbf{H}_0$ is eliminated. The other terms $\mathbf{H}_j$ will be analyzed according to the types of resonances and the coefficients $\mu_1$, $\mu_2,\ldots$ that can be found through the boundary condition of characteristic equation coefficients of the autonomous and $\tau$ periodic Hamiltonian $K$, obtained through the Deprit-Hori's method.

\begin{enumerate}
	\item {\bf A pair of resonance both even: $2\omega_1(\mu_4^*)=N_1=2$ and $2\omega_2(\mu_4^*)=N_2=4$}.
	
	For $\mu_4^*=3$ we obtain a pair of resonance $2\omega_1(\mu_4^*)=N_1=2$ and $2\omega_2(\mu_4^*)=N_2=4$. In this case, the Hamiltonian of \eqref{hmuexp} is given by
	
	\begin{eqnarray*}
		H_0&=&\frac{1}{4} \left(2(\xi_1^2+\eta_1^2)+4( \xi_2^2+ \eta_2^2)\right)\\
		H_1&=&\frac{1}{28}  \left(2 \mu_1 \left(\xi_2^2+\eta_2^2\right)+21 \cos (\nu ) \left(2 \eta_1^2-\eta_2^2\right)\right)\\
		H_2&=&-\frac{1}{392} \left(\left(5 \mu_1^2-28 \mu_2\right) \left(\xi_2^2+\eta_2^2\right)+294 \cos ^2(\nu ) \left(2 \eta_1^2-\eta_2^2\right)+35 \mu_1 \eta_2^2 \cos (\nu )\right)\\
		H_3&=&\frac{1}{10976}\Big(2 \left(\xi_2^2+\eta_2^2\right) \left(13 \mu_1^3-140 \mu_1 \mu_2+392\mu_3\right)+8232 \cos ^3(\nu ) \left(2 \eta_1^2-\eta_2^2\right)\\
		&+&7 \eta_2^2 \left(27 \mu_1^2-140 \mu_2\right) \cos (\nu )+980 \mu_1 \eta_2^2 \cos ^2(\nu )\Big).
	\end{eqnarray*}
	
	Therefore, applying \eqref{rotation} into above Hamiltonian, with $N_1=2$ and $N_2=4$, we obtain the following Hamiltonian
	\begin{equation}\label{hafterrotation}
		\mathbf{H}= \mathbf{H}_0+\epsilon \mathbf{H}_1+\dfrac{\epsilon^2}{2!}\mathbf{H}_2+\cdots+\dfrac{\epsilon^k}{k!}\mathbf{H}_k+\mathcal{O}(\epsilon^{ k+1 }),
	\end{equation}
	where
	\begin{eqnarray*}
		\mathbf{H}_0&=&0\\
		\mathbf{H}_1&=&\frac{3}{8} \left( \cos (\nu )- \cos (3 \nu )\right)X_1^2
		+\frac{3}{4}  \left(- \sin (\nu )-\sin (3 \nu )\right)X_1 Y_1\\
		&+&\frac{3}{8}  \left(3 \cos (\nu )+\cos (3 \nu )\right)Y_1^2\\
		&+& \left(\frac{\mu_1}{14}+\frac{3}{16} \cos (3 \nu )+\frac{3}{16} \cos (5 \nu )-\frac{3 \cos (\nu )}{8}\right)X_2^2\\
		&+&\frac{3}{8} \left(\sin (3 \nu )+ \sin (5 \nu )\right)X_2 Y_2\\
		&+& \left(\frac{\mu_1}{14}-\frac{3}{16} \cos (3 \nu )-\frac{3}{16} \cos (5 \nu )-\frac{3 \cos (\nu )}{8}\right)Y_2^2\\
		\mathbf{H}_2&=&\frac{3}{16}\left(\cos (4 \nu )-1\right) X_1^2 
		+\frac{3}{8} \left( 2\sin (2 \nu )+ \sin (4 \nu )\right)X_1 Y_1\\
		&+&\frac{3}{16} \left(-4 \cos (2 \nu )- \cos (4 \nu )-3\right)Y_1^2\\
		&+& \Big(-\frac{5 \mu_1^2}{392}-\frac{5}{112} \mu_1 \cos (\nu )+\frac{5}{224} \mu_1 \cos (3 \nu )+\frac{5}{224} \mu_1 \cos (5 \nu )+\frac{\mu_2}{14}\\
		&+&\frac{3}{32} \cos (2 \nu )-\frac{3}{16} \cos (4 \nu )-\frac{3}{32} \cos (6 \nu )+\frac{3}{16}\Big)X_2^2\\
		&+& \Big(\frac{5}{112} \mu_1 \sin (3 \nu )+\frac{5}{112} \mu_1 \sin (5 \nu )+\frac{1}{16} (-3) \sin (2 \nu )-\frac{3}{8} \sin (4 \nu )\\
		&-&\frac{3}{16} \sin (6 \nu )\Big)X_2 Y_2\\
		&+& \Big(-\frac{5 \mu_1^2}{392}-\frac{5}{112} \mu_1 \cos (\nu )-\frac{5}{224} \mu_1 \cos (3 \nu )-\frac{5}{224} \mu_1 \cos (5 \nu )+\frac{\mu_2}{14}\\
		&+&\frac{9}{32} \cos (2 \nu )+\frac{3}{16} \cos (4 \nu )+\frac{3}{32} \cos (6 \nu )+\frac{3}{16}\Big)Y_2^2\\
		\mathbf{H}_3&=&\frac{3}{32} \left(2 \cos (\nu )- \cos (3 \nu )- \cos (5 \nu )\right)X_1^2\\
		&+&\frac{3}{16}  \left(-2 \sin (\nu )-3 \sin (3 \nu )-\sin (5 \nu )\right)X_1 Y_1\\
		&+& \left(\frac{15 \cos (\nu )}{16}+\frac{15}{32} \cos (3 \nu )+\frac{3}{32} \cos (5 \nu )\right)Y_1^2\\
		&+& \Big(\frac{13 \mu_1^3}{5488}+\frac{27 \mu_1^2 \cos (\nu )}{3136}-\frac{27 \mu_1^2 \cos (3 \nu )}{6272}-\frac{27 \mu_1^2 \cos (5 \nu )}{6272}-\frac{5 \mu_1 \mu_2}{196}\\
		&+&\frac{5}{448} \mu_1 \cos (2 \nu )-\frac{5}{224} \mu_1 \cos (4 \nu )-\frac{5}{448} \mu_1 \cos (6 \nu )+\frac{5 \mu_1}{224}-\frac{5}{112} \mu_2 \cos (\nu )\\
		&+&\frac{5}{224} \mu_2 \cos (3 \nu )+\frac{5}{224} \mu_2 \cos (5 \nu )+\frac{\mu_3}{14}+\frac{3}{64} \cos (3 \nu )+\frac{9}{64} \cos (5 \nu )+\frac{3}{64} \cos (7 \nu )\\
		&-&\frac{15 \cos (\nu )}{64}\Big)X_2^2\\
		&+& \Big(-\frac{27 \mu_1^2 \sin (3 \nu )}{3136}-\frac{27 \mu_1^2 \sin (5 \nu )}{3136}-\frac{5}{224} \mu_1 \sin (2 \nu )-\frac{5}{112} \mu_1 \sin (4 \nu )\\
		&-&\frac{5}{224} \mu_1 \sin (6 \nu )+\frac{5}{112} \mu_2 \sin (3 \nu )+\frac{5}{112} \mu_2 \sin (5 \nu )+\frac{3 \sin (\nu )}{32}+\frac{9}{32} \sin (3 \nu )\\
		&+&\frac{9}{32} \sin (5 \nu )+\frac{3}{32} \sin (7 \nu )\Big)X_2 Y_2\\
		&+& \Big(\frac{13 \mu_1^3}{5488}+\frac{27 \mu_1^2 \cos (\nu )}{3136}+\frac{27 \mu_1^2 \cos (3 \nu )}{6272}+\frac{27 \mu_1^2 \cos (5 \nu )}{6272}-\frac{5 \mu_1 \mu_2}{196}\\
		&+&\frac{15}{448} \mu_1 \cos (2 \nu )+\frac{5}{224} \mu_1 \cos (4 \nu )+\frac{5}{448} \mu_1 \cos (6 \nu )+\frac{5 \mu_1}{224}-\frac{5}{112} \mu_2 \cos (\nu )\\
		&-&\frac{5}{224} \mu_2 \cos (3 \nu )-\frac{5}{224} \mu_2 \cos (5 \nu )+\frac{\mu_3}{14}-\frac{15}{64} \cos (3 \nu )-\frac{9}{64} \cos (5 \nu )-\frac{3}{64} \cos (7 \nu )\\
		&-&\frac{21 \cos (\nu )}{64}\Big)Y_2^2.
	\end{eqnarray*}
	Applying Deprit-Hori method, the expression for the autonomous Hamiltonian $2\pi$- periodic $K$ up to fourth order is
	\begin{eqnarray*}
		K&=&K_0+\displaystyle\sum_{j=1}^{4}\dfrac{\epsilon^j}{j!}K_j,
	\end{eqnarray*}
	where
	\begin{eqnarray*}
		K_0&=&0\\
		K_1&=&\frac{1}{14} \mu_1 \left(X_2^2+Y_2^2\right)\\
		K_2&=&\frac{3}{4}(-X_1^2+ Y_1^2)+\frac{1}{2}  \left(-\frac{5 \mu_1^2}{98}+\frac{2 \mu_2}{7}+\frac{9}{10}\right)(X_2^2+Y_2^2)\\
		K_3&=&\dfrac{3}{68600} \left(325 \mu_1^3-3500 \mu_1 \mu_2+4704 \mu_1+9800 \mu_3\right)\left(X_2^2+Y_2^2\right)\\
		K_4&=&-\frac{45 }{16}X_1^2-\frac{189 }{16}Y_1^2+\dfrac{3}{4802000}\Big(-17625 \mu_1^4+28 \mu_1^2 (9750 \mu_2-6937)-980000 \mu_1 \mu_3\\
		&-&490000 \mu_2^2+686 (1920 \mu_2+4000 \mu_4+5859)\Big)X_2^2\\
		&-&\dfrac{3}{4802000}\Big(17625 \mu_1^4-28 \mu_1^2 (9750 \mu_2-6937)+980000 \mu_1 \mu_3\\
		&+&196 \left(2500 \mu_2^2-6720 \mu_2-14000 \mu_4-84819\right)\Big)Y_2^2.
	\end{eqnarray*}
	%The Hamiltonian $K$ assumes the form
	%$$K=k_{2000}X_1^2+k_{1100}X_1X_2+k_{0200}X_2^2+k_{0020}Y_1^2+k_{0011}Y_1Y_2+k_{0002}Y_2^2.$$
	The stability and instability regions are based on the boundary conditions of the characteristic equation of the $2\pi$-periodic Hamiltonian system associated to $K$, given by $\lambda^4+a\lambda^2+b=0$, that occurs when
	\begin{equation}\label{relestab}
		a> 0\mbox{ and } b=0
		\quad
		\text{or}
		\quad
		a>0 \mbox{ and } d=a^2-4b=0.
	\end{equation}
	In this case, this coefficients represented in power series in $\epsilon$ are:
	
	\begin{eqnarray*}
		a&=&\frac{\mu_1^2}{49}\epsilon^2 +\frac{1}{3430}\left(-25 \mu_1^3+140 \mu_1 \mu_2+441 \mu_1\right)\epsilon^3
		+\frac{1}{137200} \Big(275 \mu_1^4-3000 \mu_1^2 \mu_2-462 \mu_1^2\\
		&+&5600 \mu_1 \mu_3+2800 \mu_2^2+17640 \mu_2-49392\Big)\epsilon^4+O\left(\epsilon^5\right)\\
		b&=&-\frac{9}{784}  \mu_1^2\epsilon^6+\frac{9}{54880}  \left(25 \mu_1^3-140 \mu_1 \mu_2-441 \mu_1\right)\epsilon^7+O\left(\epsilon^8\right)\\
		d&=&\frac{ \mu_1^4}{2401}\epsilon^4+\frac{1}{84035} \left(-25 \mu_1^5+140 \mu_1^3 \mu_2+441 \mu_1^3\right)\epsilon^5
		+\frac{1}{23529800} \Big(3175 \mu_1^6-35000 \mu_1^4 \mu_2\\
		&-&47334 \mu_1^4+39200 \mu_1^3 \mu_3+58800 \mu_1^2 \mu_2^2+370440 \mu_1^2 \mu_2+1123668 \mu_1^2\Big)\epsilon^6+O\left(\epsilon^7\right).
	\end{eqnarray*}

	From the conditions $b=0$ and $d=0$ we get $\mu_1=0$. The expressions of $a$, $b$ and $d$ evaluated  for $\mu_1=0$ reduce to
	
	\begin{dgroup}\label{coefa}
		\begin{dmath*}
			a=\frac{1}{2450}(5 \mu_2+42) (10 \mu_2-21)\epsilon^4+
			\frac{\mu_3}{490} (20 \mu_2+63) \epsilon^5
			+\mathcal{O}(\epsilon^6)
		\end{dmath*}
		\begin{dmath*}
			b=-\frac{9}{313600} (20 \mu_2+63)^2\epsilon^8
			-\frac{9\mu_3}{7840} (20 \mu_2+63) \epsilon^9
			+\mathcal{O}(\epsilon^{10})
		\end{dmath*}
		\begin{dmath*}
			d=\frac{1}{96040000}\left(200 \mu_2^2+1260 \mu_2+7497\right)^2\epsilon^8
			+\frac{ \mu_3}{2401000}(20 \mu_2+63) \left(200 \mu_2^2+1260 \mu_2+7497\right)\epsilon^{9}
			+\mathcal{O}(\epsilon^{10}).
		\end{dmath*}
	\end{dgroup}
	
	From $b=0$ we get the curves
	
	$$
	\begin{aligned}
		\mu_4^1 &=  3-\dfrac{63}{20}\;\epsilon^2+\frac{14553}{8000 }\;\epsilon^4\\
		\mu_4^2 &= 3-\dfrac{63}{20}\;\epsilon^2-\frac{22197}{8000}\; \epsilon^4.
	\end{aligned}
	$$

	\begin{figure}[!htb]
		\centering
		\includegraphics[scale=1]{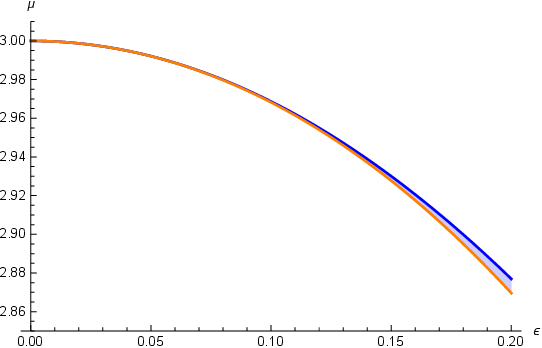}
		\caption{\textmd{Boundary curves for $b=0$.}}
		\label{colinear}
	\end{figure}

	If the condition $a>0$ were satisfied, then the shaded region in the  \ref{colinear} would be the instability region. However, we can see from the coefficient of the fourth order in $\epsilon$ of $a$ in \eqref{coefa} that $a$ can be positive only if $\mu_2<-\frac{42}{5}$ or $\mu_2>\frac{21}{10}$ and  $\mu_2=-\frac{63}{20}$, obtained from $b=0$, does not satisfy these two conditions.
	
	The second condition of \eqref{relestab}, $d=0$ and $a>0$, there is no real solution.

	\vspace{.5cm}

	\item {\bf For $\mu_3^*$ and $\mu_5^*$ both resonances are different where one is even and the other one is odd}, that is: $2\omega_1(\mu_3^*)=2$, $2\omega_2(\mu_3^*)=3$ and $2\omega_1(\mu_5^*)=2$, $2\omega_2(\mu_5^*)=5$ 
	
	\vspace{1cm}
	
	2.1\; In the first case, $N_1=2$ and $N_2=3$, the Hamiltonian of \eqref{hmuexp} is given by
	\begin{eqnarray*}
		H_0&=&\frac{1}{4} \left( 2(\xi_1^2+ \eta_1^2)+3(\xi_2^2+\eta_2^2)\right)\\
		H_1&=&\frac{1}{2688}\left(529 \mu_1 \left(\xi_2^2+\eta_2^2\right)+224 \cos (\nu ) \left(18 \eta_1^2-5 \eta_2^2\right)\right)\\
		H_2&=&-\frac{1}{10838016}\Big(529 \left(1357 \mu_1^2-4032\mu_2\right) \left(\xi_2^2+\eta_2^2\right)+903168 \cos ^2(\nu ) \left(18 \eta_1^2-5 \eta_2^2\right)\\
		&+&3080896 \mu_1 \eta_2^2 \cos (\nu )\Big)\\
		H_3&=&\frac{1}{21849440256} \Big(529 \left(\xi_2^2+\eta_2^2\right) \left(1060645 \mu_1^3-5471424 \mu_1\mu_2+8128512 \mu_3\right)\\
		&+&1820786688 \cos ^3(\nu ) \left(18 \eta_1^2-5 \eta_2^2\right)+355488 \eta_2^2 \left(7291 \mu_1^2-17472\mu_2\right) \cos (\nu )\\
		&+&6211086336 \mu_1 \eta_2^2 \cos ^2(\nu )\Big).
	\end{eqnarray*}
	
	Then, applying \eqref{rotation} into above Hamiltonian, with $N_1=2$ and $N_2=3$, we obtain the following Hamiltonian

	\begin{eqnarray*}
		\mathbf{H}_0&=&0\\
		\mathbf{H}_1&=&\left(\frac{3}{8} \cos (\nu )-\frac{3}{8} \cos (3 \nu )\right)X_1^2 
		+\left(-\frac{3}{4}  \sin (\nu )-\frac{3}{4} \sin (3 \nu )\right)X_1 Y_1 \\
		&+&\left(\frac{9}{8} \cos (\nu )+\frac{3}{8} \cos (3 \nu )\right)Y_1^2 \\	
		&+& \left(\frac{529 \mu_1}{2688}+\frac{5}{48} \cos (2 \nu )+\frac{5}{48} \cos (4 \nu )-\frac{5 \cos (\nu )}{24}\right)X_2^2\\
		&+&\left(\frac{5}{24} \sin (2 \nu )+\frac{5}{24} \sin (4 \nu )\right)X_2 Y_2 \\
		&+&\left(\frac{529 \mu_1}{2688}-\frac{5}{48} \cos (2 \nu )-\frac{5}{48} \cos (4 \nu )-\frac{5 \cos (\nu )}{24}\right)Y_2^2\\
		\mathbf{H}_2&=& \left(\frac{3}{16} \cos (4 \nu )-\frac{3}{16}\right)X_1^2
		+ \left(\frac{3}{4} \sin (2 \nu )+\frac{3}{8} \sin (4 \nu )\right)X_1 Y_1\\
		&+& \left(-\frac{3}{4}  \cos (2 \nu )-\frac{3}{16} \cos (4 \nu )-\frac{9}{16}\right)Y_1^2\\
		&+&\Big(-\frac{717853 \mu_1^2}{10838016}+\frac{6877 \mu_1 \cos (2 \nu )}{96768}+\frac{6877 \mu_1 \cos (4 \nu )}{96768}-\frac{6877 \mu_1 \cos (\nu )}{48384}\\
		&+&\frac{529 \mu_2}{2688}+\frac{5}{48} \cos (2 \nu )-\frac{5}{48} \cos (3 \nu )-\frac{5}{96} \cos (5 \nu )-\frac{5 \cos (\nu )}{96}+\frac{5}{48}\Big)X_2^2 \\
		&+&\left(\frac{6877 \mu_1 \sin (2 \nu )}{48384}+\frac{6877 \mu_1 \sin (4 \nu )}{48384}-\frac{1}{48} 5 \sin (\nu )-\frac{5}{24} \sin (3 \nu )-\frac{5}{48} \sin (5 \nu )\right)X_2 Y_2 \\
		&+& \Big(-\frac{717853 \mu_1^2}{10838016}-\frac{6877 \mu_1 \cos (\nu )}{48384}-\frac{6877 \mu_1 \cos (2 \nu )}{96768}-\frac{6877 \mu_1 \cos (4 \nu )}{96768}\\
		&+&\frac{529 \mu_2}{2688}+\frac{5 \cos (\nu )}{96}+\frac{5}{48} \cos (2 \nu )+\frac{5}{48} \cos (3 \nu )+\frac{5}{96} \cos (5 \nu )+\frac{5}{48}\Big)Y_2^2
	\end{eqnarray*}
	
	\begin{eqnarray*}
		\mathbf{H}_3&=& \left(\frac{3 \cos (\nu )}{16}-\frac{3}{32} \cos (3 \nu )-\frac{3}{32} \cos (5 \nu )\right)X_1^2\\
		&+&\left(-\frac{1}{8} 3 \sin (\nu )-\frac{9}{16} \sin (3 \nu )-\frac{3}{16} \sin (5 \nu )\right)X_1 Y_1 \\
		&+&\left(\frac{15 \cos (\nu )}{16}+\frac{15}{32} \cos (3 \nu )+\frac{3}{32} \cos (5 \nu )\right)Y_1^2 \\
		&+&\Big(\frac{561081205 \mu_1^3}{21849440256}+\frac{3856939 \mu_1^2 \cos (\nu )}{65028096}-\frac{3856939 \mu_1^2 \cos (2 \nu )}{130056192}-\frac{3856939 \mu_1^2 \cos (4 \nu )}{130056192}\\
		&-&\frac{717853 \mu_1 \mu_2}{5419008}+\frac{6877 \mu_1 \cos (2 \nu )}{96768}-\frac{6877 \mu_1 \cos (3 \nu )}{96768}-\frac{6877 \mu_1 \cos (\nu )}{193536}-\frac{6877 \mu_1 \cos (5 \nu )}{193536}\\
		&+&\frac{6877 \mu_1}{96768}+\frac{6877 \mu_2 \cos (2 \nu )}{96768}+\frac{6877 \mu_2 \cos (4 \nu )}{96768}-\frac{6877 \mu_2 \cos (\nu )}{48384}+\frac{529 \mu_3}{2688}\\
		&+&\frac{5}{64} \cos (2 \nu )-\frac{5}{96} \cos (3 \nu )+\frac{5}{64} \cos (4 \nu )+\frac{5}{192} \cos (6 \nu )-\frac{5 \cos (\nu )}{32}+\frac{5}{192}\Big)X_2^2\\
		&+&\Big(-\frac{3856939 \mu_1^2 \sin (2 \nu )}{65028096}
		-\frac{3856939 \mu_1^2 \sin (4 \nu )}{65028096}-\frac{6877 \mu_1 \sin (3 \nu )}{48384}-\frac{6877 \mu_1 \sin (\nu )}{96768}\\
		&-&\frac{6877 \mu_1 \sin (5 \nu )}{96768}+\frac{6877 \mu_2 \sin (2 \nu )}{48384}+\frac{6877 \mu_2 \sin (4 \nu )}{48384}+\frac{5}{32} \sin (2 \nu )+\frac{5}{32} \sin (4 \nu )\\
		&+&\frac{5}{96} \sin (6 \nu )\Big)X_2 Y_2 \\
		&+& \Big(\frac{561081205 \mu_1^3}{21849440256}+\frac{3856939 \mu_1^2 \cos (\nu )}{65028096}+\frac{3856939 \mu_1^2 \cos (2 \nu )}{130056192}+\frac{3856939 \mu_1^2 \cos (4 \nu )}{130056192}\\
		&-&\frac{717853 \mu_1 \mu_2}{5419008}+\frac{6877 \mu_1 \cos (\nu )}{193536}+\frac{6877 \mu_1 \cos (2 \nu )}{96768}+\frac{6877 \mu_1 \cos (3 \nu )}{96768}+\frac{6877 \mu_1 \cos (5 \nu )}{193536}\\
		&+&\frac{6877 \mu_1}{96768}-\frac{6877 \mu_2 \cos (\nu )}{48384}-\frac{6877 \mu_2 \cos (2 \nu )}{96768}-\frac{6877 \mu_2 \cos (4 \nu )}{96768}+\frac{529 \mu_3}{2688}\\
		&-&\frac{5}{64} \cos (2 \nu )-\frac{5}{96} \cos (3 \nu )-\frac{5}{64} \cos (4 \nu )-\frac{5}{192} \cos (6 \nu )-\frac{5 \cos (\nu )}{32}-\frac{5}{192}\Big)Y_2^2.
	\end{eqnarray*}
	
	Applying Deprit-Hori method, the expression for the autonomous Hamiltonian $4\pi$- periodic $K$ up to third order is
	
	\begin{equation}\label{K} K=\displaystyle\sum_{j=0}^{3}\dfrac{\epsilon^j}{j!}K_j,
	\end{equation}	
	where each coefficient is given by
	\begin{eqnarray*}
		K_0&=&0\\
		K_1&=&\frac{529}{2688} \mu_1 \left(X_2^2+Y_2^2\right)\\
		K_2&=&\dfrac{1}{2}\left(\frac{3}{4}(-X_1^2+Y_1^2) +\frac{1}{5419008}\left(1008 (2116 \mu_2+1295)-717853 \mu_1^2\right)\right)(X_2^2+Y_2^2)\\
		K_3&=&\frac{1}{21849440256}
		\Big[X_2^2 \Big(561081205 \mu_1^3+393736287 \mu_1^2-133308 \mu_1 (21712 \mu_2-16163)\\
		&+&254016 (16928 \mu_3+5915)\Big)\\
		&+&Y_2^2 \Big(561081205 \mu_1^3-393736287 \mu_1^2-133308 \mu_1 (21712 \mu_2-16163)\\
		&+&254016 (16928 \mu_3-5915)\Big)\Big]\epsilon^3.
	\end{eqnarray*}
	Just like in previous case, we have the characteristic equation associated to $4\pi-$ periodic Hamiltonian $K$, $\lambda^4+a\lambda^2+b$, where the coefficients give us the boundary stability conditions. This coefficients are:
	
	\begin{dgroup*}
		\begin{dmath*}
			a=\frac{279841}{1806336} \mu_1^2 \epsilon ^2-\frac{529}{3641573376} \mu_1  \left(717853 \mu_1^2-1008 (2116 \mu_2+1295)\right)\epsilon ^3+\frac{23}{4195092529152} \left(10574911549 \mu_1^4-76176 \mu_1^2 (749064 \mu_2-218939)+56514060288 \mu_1 \mu_3+145152 \left(194672 \mu_2^2+238280 \mu_2-633913\right)\right) \epsilon ^4+\mathcal{O}(\epsilon^5)
		\end{dmath*}
		\begin{dmath*}
			b=-\frac{279841}{3211264} \mu_1^2 \epsilon ^6+\frac{529}{6473908224} \mu_1  \left(717853 \mu_1^2-1008 (2116 \mu_2+1295)\right)\epsilon ^7+\frac{1}{7457942274048}\left(-243222965627 \mu_1^4+1752048 \mu_1^2 (749064 \mu_2+152005)-1299823386624 \mu_1 \mu_3-145152 (2116 \mu_2+1295)^2\right)\epsilon ^8 +\mathcal{O}(\epsilon^9)
		\end{dmath*}
		\begin{dmath*}
			d=\frac{78310985281}{3262849744896} \mu_1^4 \epsilon ^4-\frac{148035889}{3288952542855168} \mu_1^3  \Big(717853 \mu_1^2-1008 (2116 \mu_2+1295)\Big)\epsilon ^5+\frac{279841}{26522113305584074752} \mu_1^2  \Big(2733186618607 \mu_1^4-12264336 \mu_1^2 (1248440 \mu_2+86681)+9098763706368 \mu_1 \mu_3+3048192 \left(4477456 \mu_2^2+5480440 \mu_2+7096033\right)\Big)\epsilon ^6+\mathcal{O}(\epsilon^7).
		\end{dmath*}
	\end{dgroup*}
	
	Note that because of boundary conditions \eqref{relestab}, both $b=0$ and $d=0$, gives us $\mu_1=0$. Then, taking $\mu_1=0$, the expressions of $a$, $b$ and $d$ reduce to
	\begin{dgroup}\label{N3conditions}
		\begin{dmath*}
			a=\frac{23}{28901376}(92 \mu_2-119) (2116 \mu_2+5327)\epsilon^4
			+\frac{529}{3612672} (2116 \mu_2+1295) \mu_3\epsilon^5+\mathcal{O}(\epsilon^6)
		\end{dmath*}
		\begin{dmath*}
			b=-\frac{(2116\mu_2+1295)^2}{51380224}\epsilon^8
			-\frac{529 (2116\mu_2+1295) \mu_3}{6422528}\epsilon^9
			+\frac{97214524672}{1657320505344}\mu_2^3+59182234720\mu_2^2-44436\mu_2 (6500352\mu_4+2366525)-7 \left(20632117248 \mu_3^2+25253867520\mu_4+6658099175\right)\epsilon^{10}
			+\mathcal{O}(\epsilon^{11})
		\end{dmath*}
		\begin{dmath*}
			d=\frac{\left(4477456\mu_2^2+5480440\mu_2+17934049\right)^2}{835289534693376}\epsilon^{8}+\frac{529}{52205595918336} \Big(9474296896\mu_2^3+17394916560\mu_2^2+45045617484\mu_2+23224593455\Big) \mu_3 \epsilon^9	+\mathcal{O}(\epsilon^{10}).
		\end{dmath*}
	\end{dgroup}
	The condition $b=0$ give us the curves
	\begin{eqnarray*}
		\mu_3^1&=&\frac{20}{23}-\frac{1295 }{2116}\epsilon ^2+\frac{5915 }{16928}\epsilon ^3-\frac{106154825 }{448524288}\epsilon ^4+\mathcal{O}(\epsilon^5)\\
		\mu_3^2&=&\frac{20}{23}-\frac{1295 }{2116}\epsilon ^2-\frac{5915 }{16928}\epsilon ^3-\frac{106154825 }{448524288}\epsilon ^4+\mathcal{O}(\epsilon^5).
	\end{eqnarray*}
	
	\begin{figure}[!htb]
		\centering
		\includegraphics[scale=1]{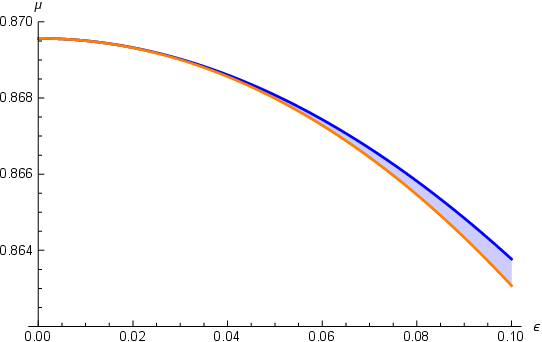}
		\caption{Boundary curves for $b=0$.}
		\label{fig3}
	\end{figure}
	
	If the condition $a>0$ were satisfied, then the shaded region in the Figure \ref{fig3} would be the instability region. However, we can see from the coefficient of the fourth order of $a$ in \eqref{N3conditions} that $a$ can be positive only if $\mu_2<-\frac{5327}{2116}$ or $\mu_2>\frac{119}{92}$ and $\mu_2=-\frac{1295}{2116}$ obtained from $b=0$ does not satisfy any of these two conditions.
	
	The second condition of \eqref{relestab}, $d=0$ and $a>0$, there is no real solution.
	
	\vspace{.5cm}
	2.2\;  In the second case, $N_1=2$ and $N_2=5$, the Hamiltonian of \eqref{hmuexp} is given by
	
	\begin{eqnarray*}
		H_0&=&\frac{1}{4} \left(2 \xi_1^2+5 \xi_2^2+2 \eta_1^2+5 \eta_2^2\right)\\
		H_1&=&\frac{1}{640}   \left(7 \mu_1 \left(\xi_2^2+\eta_2^2\right)+96 \cos (\nu ) \left(10 \eta_1^2-7 \eta_2^2\right)\right)\\
		H_2&=&-\frac{1}{1024000} \Big(7 \left(107 \mu_1^2-1600 \mu_2\right) \left(\xi_2^2+\eta_2^2\right)+153600 \cos ^2(\nu ) \left(10 \eta_1^2-7 \eta_2^2\right)\\
		&+&12992 \mu_1 \eta_2^2 \cos (\nu )\Big)\\
		H_3&=&\frac{1}{819200000}\Big(7 \left(\xi_2^2+\eta_2^2\right) \left(5749 \mu_1^3-171200 \mu_1 \mu_2+1280000 \mu_3\right)\\
		&+&122880000 \cos ^3(\nu ) \left(10 \eta_1^2-7 \eta_2^2\right)+224 \eta_2^2 \left(3159 \mu_1^2-46400 \mu_2\right) \cos (\nu )\\
		&+&10393600 \mu_1 \eta_2^2 \cos ^2(\nu )\Big).
	\end{eqnarray*}
	
	Then, applying \eqref{rotation} into above Hamiltonian, with $N_1=2$ and $N_2=5$, we obtain the following Hamiltonian

	\begin{eqnarray*}
		\mathbf{H}_0&=&0\\
		\mathbf{H}_1&=&\left(\frac{3 \cos (\nu )}{8}-\frac{3}{8} \cos (3 \nu )\right)X_1^2 
		+ \left(-\frac{1}{4} 3 \sin (\nu )-\frac{3}{4} \sin (3 \nu )\right)X_1 Y_1\\
		&+& \left(\frac{9 \cos (\nu )}{8}+\frac{3}{8} \cos (3 \nu )\right)Y_1^2\\
		&+& \left(\frac{7 \mu_1}{640}+\frac{21}{80} \cos (4 \nu )+\frac{21}{80} \cos (6 \nu )-\frac{21 \cos (\nu )}{40}\right)X_2^2\\
		&+& \left(\frac{21}{40} \sin (4 \nu )+\frac{21}{40} \sin (6 \nu )\right)X_2 Y_2\\
		&+& \left(\frac{7 \mu_1}{640}-\frac{21}{80} \cos (4 \nu )-\frac{21}{80} \cos (6 \nu )-\frac{21 \cos (\nu )}{40}\right)Y_2^2\\
		\mathbf{H}_2&=& \left(\frac{3}{16} \cos (4 \nu )-\frac{3}{16}\right)X_1^2
		+\left(\frac{3}{4} \sin (2 \nu )+\frac{3}{8} \sin (4 \nu )\right)X_1 Y_1 \\
		&+& \left(-\frac{1}{4} 3 \cos (2 \nu )-\frac{3}{16} \cos (4 \nu )-\frac{9}{16}\right)Y_1^2\\
		&+& \Big(-\frac{749 \mu_1^2}{1024000}+\frac{203 \mu_1 \cos (4 \nu )}{64000}+\frac{203 \mu_1 \cos (6 \nu )}{64000}-\frac{203 \mu_1 \cos (\nu )}{32000}+\frac{7 \mu_2}{640}\\
		&+&\frac{21}{80} \cos (2 \nu )-\frac{21}{160} \cos (3 \nu )
		-\frac{21}{80} \cos (5 \nu )-\frac{21}{160} \cos (7 \nu )+\frac{21}{80}\Big)X_2^2\\
		&+& \Big(\frac{203 \mu_1 \sin (4 \nu )}{32000}+\frac{203 \mu_1 \sin (6 \nu )}{32000}+\frac{1}{80} (-21) \sin (3 \nu )-\frac{21}{40} \sin (5 \nu )\\
		&-&\frac{21}{80} \sin (7 \nu )\Big)X_2 Y_2\\
		&+& \Big(-\frac{749 \mu_1^2}{1024000}-\frac{203 \mu_1 \cos (\nu )}{32000}-\frac{203 \mu_1 \cos (4 \nu )}{64000}-\frac{203 \mu_1 \cos (6 \nu )}{64000}+\frac{7 \mu_2}{640}\\
		&+&\frac{21}{80} \cos (2 \nu )+\frac{21}{160} \cos (3 \nu )
		+\frac{21}{80} \cos (5 \nu )+\frac{21}{160} \cos (7 \nu )+\frac{21}{80}\Big)Y_2^2\\
		\mathbf{H}_3&=& \left(\frac{3}{16} \cos (\nu )-\frac{3}{32} \cos (3 \nu )-\frac{3}{32} \cos (5 \nu )\right)X_1^2\\
		&+&\left(-\frac{3}{8} \sin (\nu )-\frac{9}{16} \sin (3 \nu )-\frac{3}{16} \sin (5 \nu )\right)X_1 Y_1 \\
		&+& \left(\frac{15 \cos (\nu )}{16}+\frac{15}{32} \cos (3 \nu )+\frac{3}{32} \cos (5 \nu )\right)Y_1^2\\
		&+& \Big(\frac{40243 \mu_1^3}{819200000}+\frac{22113 \mu_1^2 \cos (\nu )}{51200000}-\frac{22113 \mu_1^2 \cos (4 \nu )}{102400000}-\frac{22113 \mu_1^2 \cos (6 \nu )}{102400000}\\
		&-&\frac{749 \mu_1 \mu_2}{512000}+\frac{203 \mu_1 \cos (2 \nu )}{64000}-\frac{203 \mu_1 \cos (5 \nu )}{64000}-\frac{203 \mu_1 \cos (3 \nu )}{128000}-\frac{203 \mu_1 \cos (7 \nu )}{128000}\\
		&+&\frac{203 \mu_1}{64000}+\frac{203 \mu_2 \cos (4 \nu )}{64000}+\frac{203 \mu_2 \cos (6 \nu )}{64000}-\frac{203 \mu_2 \cos (\nu )}{32000}+\frac{7 \mu_3}{640}+\frac{21}{320} \cos (2 \nu )\\
		&-&\frac{21}{160} \cos (3 \nu )+\frac{63}{320} \cos (4 \nu )+\frac{63}{320} \cos (6 \nu )+\frac{21}{320} \cos (8 \nu )-\frac{63 \cos (\nu )}{160}\Big)X_2^2\\
		&+& \Big(-\frac{22113 \mu_1^2 \sin (4 \nu )}{51200000}-\frac{22113 \mu_1^2 \sin (6 \nu )}{51200000}-\frac{203 \mu_1 \sin (5 \nu )}{32000}-\frac{203 \mu_1 \sin (3 \nu )}{64000}\\
		&-&\frac{203 \mu_1 \sin (7 \nu )}{64000}+\frac{203 \mu_2 \sin (4 \nu )}{32000}+\frac{203 \mu_2 \sin (6 \nu )}{32000}+\frac{21}{160} \sin (2 \nu )+\frac{63}{160} \sin (4 \nu )\\
		&+&\frac{63}{160} \sin (6 \nu )+\frac{21}{160} \sin (8 \nu )\Big)X_2 Y_2\\
		&+& \Big(\frac{40243 \mu_1^3}{819200000}+\frac{22113 \mu_1^2 \cos (\nu )}{51200000}+\frac{22113 \mu_1^2 \cos (4 \nu )}{102400000}+\frac{22113 \mu_1^2 \cos (6 \nu )}{102400000}\\
		&-&\frac{749 \mu_1 \mu_2}{512000}+\frac{203 \mu_1 \cos (2 \nu )}{64000}+\frac{203 \mu_1 \cos (3 \nu )}{128000}+\frac{203 \mu_1 \cos (5 \nu )}{64000}+\frac{203 \mu_1 \cos (7 \nu )}{128000}\\
		&+&\frac{203 \mu_1}{64000}-\frac{203 \mu_2 \cos (\nu )}{32000}-\frac{203 \mu_2 \cos (4 \nu )}{64000}-\frac{203 \mu_2 \cos (6 \nu )}{64000}+\frac{7 \mu_3}{640}-\frac{21}{320} \cos (2 \nu )\\
		&-&\frac{21}{160} \cos (3 \nu )-\frac{63}{320} \cos (4 \nu )-\frac{63}{320} \cos (6 \nu )-\frac{21}{320} \cos (8 \nu )-\frac{63 \cos (\nu )}{160}\Big)Y_2^2.
	\end{eqnarray*}
	Applying Deprit-Hori method, the expression for the autonomous Hamiltonian $4\pi$- periodic $K$ up to fourth order is$K$ up to 
	\begin{eqnarray*}
		K&=&K_0+\displaystyle\sum_{j=1}^{4}\dfrac{\epsilon^j}{j!}K_j,
	\end{eqnarray*}
	where
	\begin{eqnarray*}
		K_0&=&0\\
		K_1&=&\frac{7 \mu_1}{640}\left(X_2^2+Y_2^2\right)\\
		K_2&=&\dfrac{1}{2}\left[\frac{3}{4}\left(-X_1^2+Y_1^2\right)+\frac{7(-107\mu_1^2+400(117+4\mu_2))}{512000}\left(X_2^2+Y_2^2\right)\right]\\
		K_3&=&\frac{7}{819200000}\left[5749\mu_1^3-100\mu_1(-5973+1712\mu_2)+1280000\mu_3\right](X_2^2+Y_2^2).\\
	\end{eqnarray*}
	Just like in previous case, we have the characteristic equation associated to $4\pi-$ periodic Hamiltonian $K$, $\lambda^4+a\lambda^2+b$, where the coefficients give us the boundary stability conditions. This coefficients are:
	\begin{eqnarray*}
		a&=&\frac{49}{102400}\mu_1^2 \epsilon^2-\frac{49}{81920000} \mu_1 \left(107 \mu_1^2-400 (4 \mu_2+117)\right)\epsilon^3\\
		&+&\frac{1}{52428800000}\Big[337561 \mu_1^4-11760 \mu_1^2 (856 \mu_2+6355)+50176000 \mu_1 \mu_3\\
		&+&32000 \left(784 \mu_2^2+45864 \mu_2-250839\right)\Big] \epsilon^4  +\mathcal{O}(\epsilon^5)\\
		b&=&-\frac{441}{1638400} \mu_1^2\epsilon^6+\frac{441}{1310720000} \mu_1 \left(107 \mu_1^2-400 (4 \mu_2+117)\right)\epsilon^7\\
		&-&\frac{441}{838860800000} \Big(6889 \mu_1^4-80 \mu_1^2 (2568 \mu_2+25465)+1024000 \mu_1 \mu_3\\
		&+&32000 (4 \mu_2+117)^2\Big)\epsilon^8+\mathcal{O}(\epsilon^9)\\
		d&=& \frac{2401}{10485760000} \mu_1^4\epsilon^4-\frac{2401}{4194304000000} \mu_1^3 \left(107 \mu_1^2-400 (4 \mu_2+117)\right)\epsilon^5\\
		&+&\frac{49}{13421772800000000} \mu_1^2 \Big[2809807 \mu_1^4-19600\mu_1^2 (4280 \mu_2+69141)\\
		&+&250880000 \mu_1\mu_3+480000 \Big(784 \mu_2^2+45864 \mu_2+977961\Big)\Big]\epsilon^6+\mathcal{O}(\epsilon^{7}).
	\end{eqnarray*}
	Note that because of boundary conditions \eqref{relestab}, both $b=0$ and $d=0$, gives us $\mu_1=0$. Then, taking $\mu_1=0$, the expressions of $a$, $b$ and $d$ reduce to

	\begin{dgroup}\label{N5conditions}
		\begin{dmath*}
			a=\frac{1}{1638400}(28 \mu_2-141) (28 \mu_2+1779)\epsilon^4+\frac{49}{204800} (4 \mu_2+117) \mu_3\epsilon^5+\mathcal{O}(\epsilon^6)
		\end{dmath*}
		\begin{dmath*}
			b=-\frac{441}{26214400} (4 \mu_2+117)^2\epsilon^8-\frac{441}{3276800} (4 \mu_2+117) \mu_3\epsilon^9
			+\frac{441}{335544320000} \Big(27392 \mu_2^3+814880 \mu_2^2-4 \mu_2 (102400 \mu_4+2117649)
			-204800 \mu_3^2-11980800 \mu_4-259455339\Big)\epsilon^{10}+\mathcal{O}(\epsilon^{11})
		\end{dmath*}
		\begin{dmath*}
			d=\frac{1}{2684354560000}\left(784 \mu_2^2+45864 \mu_2+1592361\right)^2\epsilon^8\\
			+\frac{49}{167772160000} \left(3136 \mu_2^3+275184 \mu_2^2+11735532 \mu_2+186306237\right) \mu_3 \epsilon^{9}+\mathcal{O}(\epsilon^{10}).
		\end{dmath*}
	\end{dgroup}
	From the condition $b=0$ we obtain:
	
	\begin{eqnarray*}
		\mu_5^2&=&12-\frac{117}{4}\epsilon ^2+\frac{3541407}{102400} \epsilon ^4+\frac{924075}{32768} \epsilon ^5-\frac{11716785771}{262144000} \epsilon ^6,\\
		\mu_5^2&=&12-\frac{117}{4}\epsilon ^2+\frac{3541407}{102400} \epsilon ^4-\frac{924075}{32768} \epsilon ^5-\frac{11716785771}{262144000} \epsilon^6.
	\end{eqnarray*}

	\begin{figure}[!htb]
		\centering
		\includegraphics[scale=1]{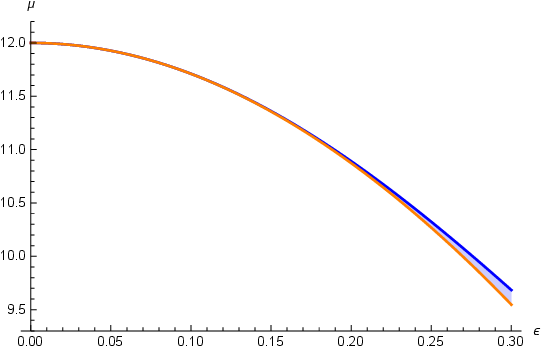}
		\caption{Boundary curves for $b=0$.}
		\label{fig5}
	\end{figure}
	
	As well as in the previous cases, if the condition $a>0$ were satisfied, then the shaded region in the Figure \ref{fig5} would be the instability region. However, we can see from the coefficient of the fourth order of $a$ in \eqref{N5conditions}, that $a$ can be positive only if $\mu_2<-\frac{1779}{28}$ or $\mu_2>\frac{141}{28}$ and $\mu_2=-\frac{117}{4}$ obtained from $b=0$ does not satisfy any of these two conditions.

	The second condition of \eqref{relestab}, $d=0$ and $a>0$, there is no real solution.
	
\end{enumerate}

\section{Conclusions}

The study of a linear and parametric stability of a subsystem of the isosceles problem generate by a perturbation of Euler's collinear solution was presented. In first moment, after performing the perturbation in the Euler's collinear solution and obtaining the spatial isosceles solution, where the masses of the base are located symmetrically with respect to the fixed line through the center of mass and parallel to the angular momentum, the circumference \eqref{cir-eq} of equilibria was found. In this situation, each equilibrium of the Hamiltonian system \eqref{hinicial} situated on this circumference corresponds to an elliptical collinear solution. Then, when performing the study of the linear stability of these equilibria, it was detected the presence of a double-zero eigenvalue resulting from the first integral $ \mathcal{Q}= -\langle \Sigma x, y \rangle$. This situation was circumvented by the Proposition \ref{dfr} that allows to reduce the degrees of freedom of this system from three to two, where now in this new simplified system, the value of $ \mathcal{Q}$ is constant and equal to $ \gamma$. With this, the second moment of this article is dedicated to the study of this simplified system, that presents a single equilibrium point when $ \gamma=1$. After obtaining the normal form of \eqref{h-red-linearizado} for $\epsilon=0$, we conclude by Dirichlet’s theorem that this equilibrium is stable. Now, as the system depends of parameters such as eccentricity $\epsilon $ and mass $\mu$, the linear stability can be affected by the presence of resonance in the system. Through Krein-Gelfand-Lidskii's theorem, it was possible found a countless of basic resonances and three double resonance. After utilizing the Deprit-Hori method in Kamel formulation, it was possible to get the autonomous and periodic Hamiltonian $K$ in each case and to discuss the problem of parametric stability in the cases of double resonance. Since the characteristic polynomial of $K$ is the form $p(\lambda)=\lambda^4+a\lambda^2+b$, the boundary conditions are as \eqref{relestab}, where in three cases only the condition $b=0$ allowed to build the boundary curves, in the form $ \mu=\mu_0^*+\mu_1\epsilon+\mu_2\epsilon^2+\mathcal{O}(\epsilon^3)$, which delimit the regions of stability and instability in the parametric plane $(\epsilon,\mu)$, however for such coefficients $ \mu_i$, $i=1,2\ldots$ obtained in each case, the condition $a>0$ is not satisfied and then it cannot be guaranteed that the shaded regions are of instability.

\section{Acknowledgments}

This work is result of PhD research supported by a Conselho Nacional de Desenvolvimento Científico e Tecnológico (CNPq) scholarship,  supervised by Hildeberto Cabral to whom I am very grateful for his many valuable contributions. I am also thanks to Adecarlos Carvalho for the implemention of the calculations.

\end{document}